\documentclass[11 pt]{amsart}
\usepackage{latexsym,amscd,amssymb, graphicx, amsthm, bm}  
\usepackage{young}
\usepackage{arydshln}
\usepackage{tikz}

\usepackage[margin=1in]{geometry}

\newtheorem{theorem}{Theorem}[]

\theoremstyle{definition}

\newcommand{\CC}{{\mathbb {C}}}

\newcommand{\cyc}{{\mathrm {cyc}}}

\newcommand{\ch}{{\mathrm {ch}}}

\begin{document}

\title[Plethysm and a character embedding problem of Miller]
{Plethysm and a character embedding problem of Miller}

\author{Brendon Rhoades}
\address
{Department of Mathematics \newline \indent
University of California, San Diego \newline \indent
La Jolla, CA, 92093, USA}
\email{bprhoades@ucsd.edu}

\begin{abstract}
We use a plethystic formula of Littlewood  
to answer a question of Miller on embeddings of symmetric group characters.
We also reprove a result of Miller on character congruences.
\end{abstract}

\maketitle

Given $d \geq 1$ and a partition $\lambda = (1^{m_1} 2^{m_2} 3^{m_3} \cdots )$
 of a positive integer $n$, let $\boxplus^d(\lambda)$ be the partition of $d^2 \cdot n$
given by
$\boxplus^d(\lambda) := ( d^{d m_1} (2d)^{d m_2} (3d)^{d m_3} \cdots )$.
The Young diagram of $\boxplus^d(\lambda)$ is obtained from that of $\lambda$ by subdividing every box into a $d \times d$ grid, 
as suggested by the notation.

Let $S_n$ be the symmetric group on $n$ letters. For a partition $\lambda \vdash n$, let 
$V^{\lambda}$ be the corresponding $S_n$-irreducible  with character $\chi^{\lambda}: S_n \rightarrow \CC$.
 For  $d \geq 1$, define a new class function $\boxplus^d(\chi^{\lambda})$ 
on $S_n$ whose value on permutations of cycle type $\mu \vdash n$ is given by
\begin{equation}
\boxplus^d(\chi^{\lambda})_{\mu} := \chi^{\boxplus^d(\lambda)}_{\boxplus^d(\mu)}.  
\end{equation}
Thus, the values of the class function $\boxplus^d(\chi^{\lambda})$ on $S_n$ are embedded inside the character table of the larger symmetric group
$S_{d^2 \cdot n}$.  A. Miller conjectured \cite{MillerComm} that the class functions $\boxplus^d(\chi^{\lambda})$ are genuine
characters of (rather than merely class functions on) 
$S_n$. We prove that this is so in Theorem~\ref{embedding-theorem} using {\em plethysm} of symmetric functions.

In the arguments that follow, we use standard material on symmetric functions;
for details see \cite{MacdonaldBook}.
For $\mu \vdash n$, let $m_i(\mu)$ be the multiplicity of $i$ as a part of $\mu$ and  
$z_{\mu} := 1^{m_1(\mu)} 2^{m_2(\mu)} \cdots m_1(\mu)! m_2(\mu)! \cdots$ be the size of the centralizer of a permutation
$w \in S_n$ of cycle type $\mu$.

Let $\Lambda = \bigoplus_{n \geq 0} \Lambda_n$ be the ring of symmetric functions in an infinite variable set $(x_1, x_2, \dots )$.
Bases of $\Lambda$ are indexed by partitions; we use the Schur basis $\{ s_{\lambda} \}$ and power sum basis $\{ p_{\lambda} \}$.
The basis $p_{\lambda}$ is {\em multiplicative}: if $\lambda = (\lambda_1, \lambda_2, \dots )$ then $p_{\lambda} = p_{\lambda_1} p_{\lambda_2} \cdots$.
The transition matrix from the Schur to the power sum basis encodes the character table of $S_n$; for $\lambda \vdash n$ we have
$s_{\lambda} = \sum_{\mu \vdash n} \frac{\chi^{\lambda}_{\mu}}{z_{\mu}} p_{\mu}$.

Let $\langle - , - \rangle$ be the {\em Hall inner product} on $\Lambda$ with respect to which the
Schur basis $\{ s_{\lambda} \}$ is orthonormal.
The power sums are orthogonal with respect to this inner product. We have $\langle p_{\lambda}, p_{\mu} \rangle = z_{\lambda} \cdot \delta_{\lambda,\mu}$
where $\delta$ is the Kronecker delta.

Write $R = \bigoplus_{n \geq 0} R_n$ where $R_n$ 
is the space of class functions $\varphi: S_n \rightarrow \CC$.  The {\em characteristic map}  $\ch_n: R_n \rightarrow \Lambda_n$ is
given by $\ch_n(\varphi) = \frac{1}{n!} \sum_{w \in S_n} \varphi(w) \cdot p_{\cyc(w)}$ 
where $\cyc(w) \vdash n$ is the cycle type of $w \in S_n$.
The map $\ch = \bigoplus_{n \geq 0} \ch_n$ is a linear isomorphism $R \rightarrow \Lambda$.
The space $R$ has an {\em induction product} given by
$\varphi \circ \psi := \mathrm{Ind}_{S_n \times S_m}^{S_{n+m}} (\varphi \otimes \psi)$ for all $\varphi \in R_n$ and $\psi \in R_m$.
Under this product, the map $\ch: R \rightarrow \Lambda$ becomes a ring isomorphism.
We record two properties of $\ch$. 
\begin{itemize}
\item
 We have $\ch(\chi^{\lambda}) = s_{\lambda}$, so that $\ch$ sends the irreducible character basis of $R$
to the Schur basis of $\Lambda$. 
\item
If $\varphi: S_n \rightarrow \CC$ is any class function and $\mu \vdash n$, then
\begin{equation}
\label{character-interpretation}
\langle \ch(\varphi), p_{\mu} \rangle = \text{ value of $\varphi$ on a permutation of cycle type $\mu$.}
\end{equation}
\end{itemize}

Let $\psi^d: \Lambda \rightarrow \Lambda$ be the map 
$\psi^d: F(x_1, x_2, \dots ) \mapsto F(x_1^d, x_2^d, \dots )$
which replaces each variable $x_i$ with its $d^{th}$ power $x_i^d$.  The symmetric function $\psi^d(F)$ is the plethysm
$p_d[F]$ of $F$ into the power sum $p_d$.  Let $\phi_d: \Lambda \rightarrow \Lambda$ be the adjoint of $\psi^d$ 
 characterized by
$\langle \psi^d(F), G \rangle = \langle F, \phi_d(G) \rangle$  for all $F, G \in \Lambda$.
In this note we apply the operators $\psi^d$ and $\phi_d$ to character theory; see \cite{Rhoades} for an application to
the cyclic sieving phenomenon of enumerative combinatorics.

\begin{theorem}
\label{embedding-theorem}
Let $d \geq 1$ and $\lambda \vdash n$. Consider the chain of subgroups
$\Delta(S_n) \subseteq S_n^d \subseteq S_{dn}$
where $S_n^d = S_n \times \cdots \times S_n$ is the $d$-fold self-product of $S_n$ and
  $\Delta(S_n)$ is the diagonal  $\{ (w, \dots, w) \,:\, w \in S_n \}$ in $S_d^n$.  Then $\boxplus^d(\chi^{\lambda})$ is the character of the
   $\Delta(S_n) \cong S_n$ module
\begin{equation}
  \mathrm{Res}^{S_{dn}}_{\Delta(S_n)}  (V^{\lambda} \circ \cdots \circ V^{\lambda})
\end{equation}
obtained by restricting the $d$-fold induction product
$ V^{\lambda} \circ \cdots \circ V^{\lambda} = \mathrm{Ind}_{S_n^d}^{S_{dn}}(V^{\lambda} \otimes \cdots \otimes V^{\lambda})$ to $\Delta(S_n)$.
\end{theorem}

\begin{proof}
Let $\lambda, \mu \vdash n$ be two partitions and let $d \geq 1$.
By \eqref{character-interpretation} we have the class function value
\begin{equation}
\label{first-string}
\chi^{\boxplus^d(\lambda)}_{\boxplus^d(\mu)} = 
\langle s_{\boxplus^d(\lambda)}, p_{\boxplus^d(\mu)} \rangle = 
\langle s_{\boxplus^d(\lambda)}, \psi^d(p_{\mu}^d) \rangle = 
\langle \phi_d(s_{\boxplus^d(\lambda)}), p_{\mu}^d \rangle.
\end{equation}
Littlewood \cite{Littlewood} proved (see also \cite{LLT})
that for any partition $\nu \vdash dm$, the image $\phi_d(s_{\nu})$ is given by
\begin{equation}
\label{littlewood-formula}
\phi_d(s_{\nu}) = \epsilon_d(\nu) \cdot s_{\nu^{(1)}} \cdots s_{\nu^{(d)}}
\end{equation}
where $\epsilon_d(\nu)$ is the $d$-sign of $\nu$ and $(\nu^{(1)}, \dots, \nu^{(d)})$ is the $d$-quotient of $\nu$.
We refer the reader to \cite{LLT, Littlewood} for definitions. In our context we have $\epsilon_d(\boxplus^d(\lambda)) = +1$
(since $\boxplus^d(\lambda)$ admits a $d$-ribbon tiling with only horizontal ribbons) and
the $d$-quotient of $\boxplus^d(\lambda)$ is the constant $d$-tuple $(\lambda, \dots, \lambda)$.  Equation~\eqref{littlewood-formula} reads
\begin{equation}
\label{new-littlewood-formula}
\phi_d(s_{\boxplus^d(\lambda)}) = s_{\lambda}^d.
\end{equation}
Plugging \eqref{new-littlewood-formula} into \eqref{first-string} gives
\begin{equation}
\chi^{\boxplus^d(\lambda)}_{\boxplus^d(\mu)} = 
\langle \phi_d(s_{\boxplus^d(\lambda)}), p_{\mu}^d \rangle =
\langle s_{\lambda}^d, p_{\mu}^d \rangle
\end{equation}
which (thanks to~\eqref{character-interpretation})
agrees with the trace of $(w, \dots, w) \in \Delta(S_n)$ on 
$V^{\lambda} \circ \cdots \circ V^{\lambda}$
for $w \in S_n$ of cycle type $\mu$.
\end{proof}

If $\lambda = (\lambda_1, \lambda_2, \dots )$ is a partition, let $d \cdot \lambda = (d \lambda_1, d \lambda_2, \dots )$ be the partition obtained by 
multiplying every part of $\lambda$ by $d$. The argument proving Theorem~\ref{embedding-theorem} applies to show that for $\lambda \vdash n$,
the class function $\chi^{d \cdot \lambda}: S_n \rightarrow \CC$ given by $(\chi^{d \cdot \lambda})_{\mu} := \chi^{d \cdot \lambda}_{d \cdot \mu}$
is a genuine character (although its module does not have such a nice description).
It may be interesting to find other ways to discover characters of $S_n$ embedded inside characters of larger symmetric groups.

In closing, we use plethysm to give a quick proof of a character congruence result of Miller \cite[Thm. 1]{MillerCongruences}.
Miller gave an interesting combinatorial proof the following theorem by introducing objects called `cascades'.

\begin{theorem}
\label{congruence-theorem}  {\em (Miller)}
Let $d \geq 1$.
For any partitions $\lambda \vdash n$ and $\mu \vdash dn$, we have
\begin{equation}
\chi^{\boxplus^d(\lambda)}_{d \cdot \mu} \equiv 0  \mod d!.
\end{equation}
Furthermore, suppose $\lambda, \nu \vdash n$ with $d \nmid n$. Then
\begin{equation}
\chi^{\boxplus^d(\lambda)}_{d^2 \cdot \nu} = 0.
\end{equation}
\end{theorem}

\begin{proof}
Arguing as in the proof of Theorem~\ref{embedding-theorem}, we have
\begin{equation}
\chi^{\boxplus^d(\lambda)}_{d \cdot \mu} =
\langle s_{\boxplus^d(\lambda)}, p_{d \cdot \mu} \rangle = 
\langle s_{\boxplus^d(\lambda)}, \psi^d(p_{\mu}) \rangle = 
\langle \phi_d(s_{\boxplus^d(\lambda)}), p_{\mu} \rangle =
\langle s_{\lambda}^d, p_{\mu} \rangle
\end{equation}
where the last equality
 used Equation~\eqref{new-littlewood-formula}. We have $s_{\lambda} = \sum_{\rho \vdash n} \frac{\chi^{\lambda}_{\rho}}{z_{\rho}} p_{\rho}$ so that 
\begin{equation}
\label{inner-product-interesting}
\chi^{\boxplus^d(\lambda)}_{d \cdot \mu} = \langle s_{\lambda}^d, p_{\mu} \rangle =
\left\langle \left( \sum_{\rho \vdash n} \frac{\chi^{\lambda}_{\rho}}{z_{\rho}} p_{\rho} \right)^d, p_{\mu} \right\rangle.
\end{equation}
We expand far right of \eqref{inner-product-interesting} using the orthogonality of the $p$'s to obtain 
\begin{equation}
\label{hall-summation}
\left\langle \left( \sum_{\rho \vdash n} \frac{\chi^{\lambda}_{\rho}}{z_{\rho}} p_{\rho} \right)^d, p_{\mu} \right\rangle = 
\sum_{(\mu_{(1)}, \dots, \mu_{(d)})} \frac{z_{\mu}}{ z_{\mu_{(1)}} \cdots z_{\mu_{(d)}}}
 \times \chi^{\lambda}_{\mu_{(1)}} \cdots \chi^{\lambda}_{\mu_{(d)}}
\end{equation}
where the sum is over all $d$-tuples $(\mu_{(1)}, \dots, \mu_{(d)})$ of partitions of $n$ whose multiset of parts equals $\mu$.
In particular, \eqref{hall-summation} is zero unless every part of $\mu$ is $\leq n$; we assume this going forward.
We want to show that 
\eqref{hall-summation} is divisible by $d!$. To show this, we examine what happens when some of the entries in a tuple 
$(\mu_{(1)}, \dots, \mu_{(d)})$ coincide.

Fix a $d$-tuple $(\mu_{(1)}, \dots, \mu_{(d)})$ of partitions of $n$ whose multiset of parts is $\mu$.
The ratio of $z$'s in the corresponding term on the RHS of \eqref{hall-summation} is a product of multinomial coefficients
\begin{equation}
\label{multinomial-product}
 \frac{z_{\mu}}{ z_{\mu_{(1)}} \cdots z_{\mu_{(d)}}} = 
 {m_1(\mu) \choose m_1(\mu_{(1)}), \dots , m_1(\mu_{(d)})}  \cdots 
  {m_n(\mu) \choose m_n(\mu_{(1)}), \dots , m_n(\mu_{(d)})}.
\end{equation}
Let $\sigma = (\sigma_1, \dots, \sigma_r ) \vdash d$ be the partition of $d$
 obtained by writing the entry multiplicities  in the $d$-tuple $(\mu_{(1)}, \dots, \mu_{(d)})$ in weakly
decreasing order.
For example, if $n = 3$, $d = 5$, and our $d$-tuple of partitions of $n$ is
 $(\mu_{(1)}, \dots, \mu_{(5)}) = ((2,1), (3), (1,1,1), (3), (2,1))$, then $\sigma = (2,2,1)$.
Each multinomial coefficient in \eqref{multinomial-product} for which $m_i(\mu) > 0$ is divisible by 
$\sigma_1! \cdots \sigma_r!$. Since each part of $\mu$ is $\leq n$, at least one $m_i(\mu) > 0$ and
the whole product \eqref{multinomial-product} of multinomial coefficients is divisible by $\sigma_1! \cdots \sigma_r!$.
Thus, the sum of the terms in \eqref{hall-summation} indexed by rearrangements of $(\mu_{(1)}, \dots, \mu_{(d)})$ is divisible by 
${d \choose \sigma_1, \dots, \sigma_r} \cdot \sigma_1! \cdots \sigma_r! = d!$, so that \eqref{hall-summation} 
itself is divisible by $d!$. This proves the first part 
of the theorem.

For the second part of the theorem, let $\lambda, \nu \vdash n$ where $d \nmid n$. Arguing as above, we have
\begin{equation}
\label{vanish-equation}
\chi^{\boxplus^d(\lambda)}_{d^2 \cdot \nu}  = 
\left\langle \left( \sum_{\rho \vdash n} \frac{\chi^{\lambda}_{\rho}}{z_{\rho}} p_{\rho} \right)^d, p_{d \cdot \nu} \right\rangle.
\end{equation}
Since $d \nmid n$, each partition $\rho \vdash n$ appearing in the first argument of the inner product in \eqref{vanish-equation}
has at least one part not divisible by $d$. 
 Since the $p$'s are an orthogonal basis of $\Lambda$, we see that \eqref{vanish-equation} $ = 0$, proving 
the second part of the theorem.
\end{proof}

\section*{Acknowledgements}
\label{Acknowledgements}

B. Rhoades was partially supported by NSF Grant DMS-1953781 and is grateful to Alex Miller for many helpful conversations, as
well as comments on a draft of this manuscript.

\end{document}